\documentclass{amsart}
\usepackage[dvips]{graphicx}
\usepackage{amssymb}

\newtheorem{theorem}{Theorem}[section]

\newtheorem{lemma}[theorem]{Lemma}
\newtheorem{proposition}[theorem]{Proposition}

\theoremstyle{remark}
\newtheorem{definition}[theorem]{\sc Definition}
\newtheorem{remark}[theorem]{\sc Remark}
\newtheorem{example}[theorem]{\sc Example}
\newtheorem*{acknowledgment}{\sc Acknowledgment}

\numberwithin{equation}{section}

\begin{document}

\title{Toric weak Fano varieties associated to building sets}

\author{Yusuke Suyama}
\address{Department of Mathematics, Graduate School of Science, Osaka City University,
3-3-138 Sugimoto, Sumiyoshi-ku, Osaka 558-8585 JAPAN}
\email{d15san0w03@st.osaka-cu.ac.jp}

\subjclass[2010]{Primary 14M25; Secondary 14J45, 05C20.}

\keywords{toric weak Fano varieties, building sets,
nested sets, reflexive polytopes, directed graphs.}

\date{\today}


\begin{abstract}
We give a necessary and sufficient condition
for the nonsingular projective toric variety associated to a building set
to be weak Fano in terms of the building set.
\end{abstract}

\maketitle

\section{Introduction}

A {\it toric Fano variety} is a nonsingular projective toric variety over $\mathbb{C}$
whose anticanonical divisor is ample.
It is known that there are a finite number of isomorphism classes
of toric Fano varieties in any given dimension.
The classification problem of toric Fano varieties
has been studied by many researchers.
In particular, {\O}bro \cite{Obro} gave an explicit algorithm
that classifies all toric Fano varieties for any dimension.

A nonsingular projective algebraic variety is said to be {\it weak Fano}
if its anticanonical divisor is nef and big.
Sato \cite{Sato02} classified toric weak Fano 3-folds
that are not Fano but are deformed to Fano under a small deformation,
which are called toric {\it weakened Fano} 3-folds.

We can construct a nonsingular projective toric variety from a building set.
Since a finite simple graph defines a building set, which is called the {\it graphical building set},
we can also associate to the graph a toric variety (see, for example \cite{Zelevinsky}).
The author \cite{Suyama1604, Suyama1611} characterized finite simple graphs
whose associated toric varieties are Fano or weak Fano,
and building sets whose associated toric varieties are Fano.
In this paper, we characterize building sets
whose associated toric varieties are weak Fano (see Theorem \ref{main}).
Our theorem is proved combinatorially by using the fact that
the intersection number of the anticanonical divisor with a torus-invariant curve
can be computed in terms of the building set.

A toric weak Fano variety defines a reflexive polytope.
Higashitani \cite{Higashitani} constructed
integral convex polytopes from finite directed graphs
and gave a necessary and sufficient condition for the polytope to be terminal and reflexive.
We also discuss a difference between the class of reflexive polytopes
defined by toric weak Fano varieties associated to building sets,
and that of reflexive polytopes associated to finite directed graphs.

The structure of the paper is as follows:
In Section 2, we review the construction of a toric variety from a building set
and state the characterization of building sets
whose associated toric varieties are weak Fano.
In Section 3, we consider reflexive polytopes associated to building sets
and finite directed graphs.
In Section 4, we give a proof of the main theorem.

\begin{acknowledgment}
This work was supported by Grant-in-Aid for JSPS Fellows 15J01000.
The author wishes to thank his supervisor, Professor Mikiya Masuda,
for his valuable advice, and Professor Akihiro Higashitani for his useful comments.
\end{acknowledgment}

\section{The main result}

Let $S$ be a nonempty finite set.
A {\it building set} on $S$ is a finite set $B$ of nonempty subsets of $S$
satisfying the following conditions:
\begin{enumerate}
\item $I, J \in B$ with $I \cap J \ne \emptyset$ implies $I \cup J \in B$.
\item $\{i\} \in B$ for every $i \in S$.
\end{enumerate}
We denote by $B_{\rm max}$ the set of all maximal (by inclusion) elements of $B$.
An element of $B_{\rm max}$ is called a {\it $B$-component}
and $B$ is said to be {\it connected} if $B_{\rm max}=\{S\}$.
For a nonempty subset $C$ of $S$, we call $B|_C=\{I \in B \mid I \subset C\}$
the {\it restriction} of $B$ to $C$. $B|_C$ is a building set on $C$.
Note that $B|_C$ is connected if and only if $C \in B$.
For any building set $B$, we have $B=\bigsqcup_{C \in B_{\rm max}} B|_C$.
In particular, any building set is a disjoint union of connected building sets.

\begin{definition}\label{nested}
Let $B$ be a building set.
A {\it nested set} of $B$ is a subset $N$ of $B \setminus B_{\rm max}$
satisfying the following conditions:
\begin{enumerate}
\item If $I, J \in N$, then we have either
$I \subset J$ or $J \subset I$ or $I \cap J=\emptyset$.
\item For any integer $k \geq 2$ and for any pairwise disjoint $I_1, \ldots, I_k \in N$,
the union $I_1 \cup \cdots \cup I_k$ is not in $B$.
\end{enumerate}
\end{definition}

Note that the empty set is a nested set for any building set. 
The set $\mathcal{N}(B)$ of all nested sets of $B$ is called the {\it nested complex}.
$\mathcal{N}(B)$ is in fact an abstract simplicial complex on $B \setminus B_{\rm max}$.

\begin{proposition}[{\cite[Proposition 4.1]{Zelevinsky}}]\label{pure}
Let $B$ be a building set on $S$.
Then every maximal (by inclusion) nested set of $B$
has the same cardinality $|S|-|B_{\rm max}|$.
In particular, if $B$ is connected,
then the cardinality of every maximal nested set of $B$ is $|S|-1$.
\end{proposition}

We are now ready to construct a toric variety from a building set.
First, suppose that $B$ is connected and $S=\{1, \ldots, n+1\}$.
We denote by $e_1, \ldots, e_n$ the standard basis for $\mathbb{R}^n$
and we put $e_{n+1}=-e_1-\cdots-e_n$.
For a nonempty subset $I$ of $S$, we denote $e_I=\sum_{i \in I}e_i$.
Note that $e_S=0$.
For $N \in \mathcal{N}(B) \setminus \{\emptyset\}$,
we denote by $\mathbb{R}_{\geq 0}N$
the $|N|$-dimensional cone $\sum_{I \in N}\mathbb{R}_{\geq 0}e_I$ in $\mathbb{R}^n$,
where $\mathbb{R}_{\geq 0}$ is the set of nonnegative real numbers,
and we define $\mathbb{R}_{\geq 0}\emptyset$ to be $\{0\} \subset \mathbb{R}^n$.
Then $\Delta(B)=\{\mathbb{R}_{\geq 0}N \mid N \in \mathcal{N}(B)\}$
forms a fan in $\mathbb{R}^n$
and thus we have an $n$-dimensional toric variety $X(\Delta(B))$.
If $B$ is not connected,
then we define $X(\Delta(B))=\prod_{C \in B_{\rm max}}X(\Delta(B|_C))$.

\begin{theorem}[{\cite[Corollary 5.2 and Theorem 6.1]{Zelevinsky}}]
Let $B$ be a building set.
Then the associated toric variety $X(\Delta(B))$ is nonsingular and projective.
\end{theorem}

The following is our main result:

\begin{theorem}\label{main}
Let $B$ be a building set. Then the following are equivalent:
\begin{enumerate}
\item The associated toric variety $X(\Delta(B))$ is weak Fano.
\item For any $B$-component $C$ and
for any $I_1, I_2 \in B|_C$ such that $I_1 \cap I_2 \ne \emptyset,
I_1 \not\subset I_2$ and $I_2 \not\subset I_1$,
we have at least one of the following:
\begin{itemize}
\item[(i)] $I_1 \cap I_2 \in B|_C$.
\item[(ii)] $I_1 \cup I_2=C$ and $|(B|_{I_1 \cap I_2})_{\rm max}| \leq 2$.
\end{itemize}
\end{enumerate}
\end{theorem}

\begin{remark}
In a previous paper \cite{Suyama1604},
we proved that the toric variety associated to a {\it graphical} building set is weak Fano
if and only if every connected component of the graph does not have
a cycle graph of length $\geq 4$ or a diamond graph as a proper induced subgraph.
However, it is unclear whether this result can be obtained from Theorem \ref{main}.
\end{remark}

\begin{example}
Theorem \ref{main} implies that if $|S| \leq 4$,
then the toric variety $X(\Delta(B))$ is weak Fano
for any connected building set $B$ on $S$.
Any building set is a disjoint union of connected building sets,
and the disjoint union corresponds to
the product of toric varieties associated to the connected building sets.
Since the product of toric weak Fano varieties is also weak Fano,
it follows that all toric varieties of dimension $\leq 3$
associated to building sets are weak Fano.
\end{example}

We recall a description of the intersection number of the anticanonical divisor
with a torus-invariant curve, see \cite{Oda} for details.
Let $\Delta$ be a nonsingular complete fan in $\mathbb{R}^n$
and let $X(\Delta)$ be the associated toric variety.
For $0 \leq r \leq n$,
we denote by $\Delta(r)$ the set of $r$-dimensional cones in $\Delta$.
For $\tau \in \Delta(n-1)$,
the intersection number of the anticanonical divisor $-K_{X(\Delta)}$ with
the torus-invariant curve $V(\tau)$ corresponding to $\tau$
can be computed as follows:

\begin{proposition}\label{intersection number}
Let $\Delta$ be a nonsingular complete fan in $\mathbb{R}^n$
and $\tau=\mathbb{R}_{\geq 0}v_1+\cdots+\mathbb{R}_{\geq 0}v_{n-1} \in \Delta(n-1)$,
where $v_1, \ldots, v_{n-1}$ are primitive vectors in $\mathbb{Z}^n$.
Let $v$ and $v'$ be the distinct primitive vectors in $\mathbb{Z}^n$ such that
$\tau+\mathbb{R}_{\geq 0}v$ and $\tau+\mathbb{R}_{\geq 0}v'$ are in $\Delta(n)$.
Then there exist unique integers $a_1, \ldots, a_{n-1}$ such that
$v+v'+a_1v_1+\cdots+a_{n-1}v_{n-1}=0$.
Furthermore, the intersection number $(-K_{X(\Delta)}.V(\tau))$ is equal to
$2+a_1+\cdots+a_{n-1}$.
\end{proposition}

\begin{proposition}[{\cite[Proposition 6.17]{Sato00}}]\label{weak Fano}
Let $X(\Delta)$ be an $n$-dimensional nonsingular projective toric variety.
Then $X(\Delta)$ is weak Fano if and only if
$(-K_{X(\Delta)}.V(\tau))$ is nonnegative for every
$(n-1)$-dimensional cone $\tau$ in $\Delta$.
\end{proposition}

\begin{example}
Let $S=\{1, 2, 3, 4, 5\}$ and
\begin{equation*}
B=\{\{1\}, \{2\}, \{3\}, \{4\}, \{5\}, \{1, 2, 3, 4\}, \{2, 3, 4, 5\}, \{1, 2, 3, 4, 5\}\}.
\end{equation*}
Then the nested complex $\mathcal{N}(B)$ consists of
\begin{align*}
&\{\{1\}, \{2\}, \{3\}, \{1, 2, 3, 4\}\}, \{\{1\}, \{2\}, \{4\}, \{1, 2, 3, 4\}\},\\
&\{\{1\}, \{3\}, \{4\}, \{1, 2, 3, 4\}\}, \{\{2\}, \{3\}, \{4\}, \{1, 2, 3, 4\}\},\\
&\{\{2\}, \{3\}, \{4\}, \{2, 3, 4, 5\}\}, \{\{2\}, \{3\}, \{5\}, \{2, 3, 4, 5\}\},\\
&\{\{2\}, \{4\}, \{5\}, \{2, 3, 4, 5\}\}, \{\{3\}, \{4\}, \{5\}, \{2, 3, 4, 5\}\},\\
&\{\{1\}, \{2\}, \{3\}, \{5\}\}, \{\{1\}, \{2\}, \{4\}, \{5\}\}, \{\{1\}, \{3\}, \{4\}, \{5\}\}
\end{align*}
and their subsets.
The pair $I_1=\{1, 2, 3, 4\}$ and $I_2=\{2, 3, 4, 5\}$
does not satisfy the condition (2) in Theorem \ref{main}.
Hence the 4-dimensional toric variety $X(\Delta(B))$ is not weak Fano.
In fact, there exists a 3-dimensional cone $\tau$ in $\Delta(B)$ such that
$(-K_{X(\Delta(B))}.V(\tau)) \leq -1$. Let
\begin{equation*}
N_1=\{\{2\}, \{3\}, \{4\}, \{1, 2, 3, 4\}\},\quad
N_2=\{\{2\}, \{3\}, \{4\}, \{2, 3, 4, 5\}\}.
\end{equation*}
Then we have
\begin{align*}
\mathbb{R}_{\geq 0}N_1&=\mathbb{R}_{\geq 0}e_2+\mathbb{R}_{\geq 0}e_3
+\mathbb{R}_{\geq 0}e_4+\mathbb{R}_{\geq 0}(e_1+e_2+e_3+e_4),\\
\mathbb{R}_{\geq 0}N_2&=\mathbb{R}_{\geq 0}e_2+\mathbb{R}_{\geq 0}e_3
+\mathbb{R}_{\geq 0}e_4+\mathbb{R}_{\geq 0}(-e_1).
\end{align*}
Let us consider $\tau=\mathbb{R}_{\geq 0}N_1 \cap \mathbb{R}_{\geq 0}N_2
=\mathbb{R}_{\geq 0}e_2+\mathbb{R}_{\geq 0}e_3+\mathbb{R}_{\geq 0}e_4$.
Since $(e_1+e_2+e_3+e_4)+(-e_1)-e_2-e_3-e_4=0$,
Proposition \ref{intersection number} gives $(-K_{X(\Delta(B))}.V(\tau))=2-3=-1$.
Therefore $X(\Delta(B))$ is not weak Fano by Proposition \ref{weak Fano}.
\end{example}

\section{Reflexive polytopes associated to building sets}

An $n$-dimensional integral convex polytope
$P \subset \mathbb{R}^n$ is said to be {\it reflexive}
if $0$ is in the interior of $P$ and the dual $P^*=\{u \in \mathbb{R}^n \mid
\langle u, v\rangle \geq -1 \mbox{ for any } v \in P\}$
is also an integral convex polytope,
where $\langle \cdot, \cdot\rangle$ denotes
the standard inner product in $\mathbb{R}^n$.
Let $\Delta$ be a nonsingular complete fan in $\mathbb{R}^n$.
If the associated toric variety $X(\Delta)$ is weak Fano,
then the convex hull of primitive generators of rays in $\Delta(1)$ is a reflexive polytope.
For a building set $B$ such that the associated toric variety $X(\Delta(B))$
is weak Fano, we denote by $P_B$ the corresponding reflexive polytope.

Higashitani \cite{Higashitani} gave a construction
of integral convex polytopes from finite directed graphs
(with no loops and no multiple arrows).
We describe his construction briefly.
Let $G$ be a finite directed graph
whose node set is $V(G)=\{1, \ldots, n+1\}$
and whose arrow set is $A(G) \subset V(G) \times V(G)$.
For $\overrightarrow{e}=(i, j) \in A(G)$,
we define $\rho(\overrightarrow{e}) \in \mathbb{R}^{n+1}$ to be $e_i-e_j$.
We define $P_G$ to be the convex hull of $\{\rho(\overrightarrow{e}) \mid
\overrightarrow{e} \in A(G)\}$ in $\mathbb{R}^{n+1}$.
$P_G$ is an integral convex polytope in the hyperplane
$H=\{(x_1, \ldots, x_{n+1}) \in \mathbb{R}^{n+1} \mid x_1+\cdots+x_{n+1}=0\}$.
In a previous paper we proved that if $X(\Delta(B))$ is Fano,
then $P_B$ can be obtained from a finite directed graph:

\begin{theorem}[{\cite[Theorem 4.1]{Suyama1611}}]
Let $B$ be a building set.
If the associated toric variety $X(\Delta(B))$ is Fano,
then there exists a finite directed graph $G$ such that $P_B$
is equivalent to $P_G$, that is,
there exists a linear isomorphism $f:\mathbb{R}^n \rightarrow H$
such that $f(\mathbb{Z}^n)=H \cap \mathbb{Z}^{n+1}$ and $f(P_B)=P_G$.
\end{theorem}

However, there exist infinitely many reflexive polytopes associated to building sets
that cannot be obtained from finite directed graphs.
The following proposition provides such examples:

\begin{proposition}
Let $S=\{1, \ldots, n+1\}$ and $B=2^S \setminus \{\emptyset\}$.
Then $X(\Delta(B))$ is weak Fano by Theorem \ref{main}
but the reflexive polytope $P_B$
cannot be obtained from any finite directed graph for $n \geq 3$.
\end{proposition}

\begin{proof}
Suppose that there exists a finite directed graph $G$
such that $P_B$ is equivalent to $P_G$.
Since $0 \in P_G$, there exists a nonempty subset $A'$ of $A(G)$
and positive real numbers $a_{\overrightarrow{e}}$
for $\overrightarrow{e} \in A'$
such that $\sum_{\overrightarrow{e} \in A'}
a_{\overrightarrow{e}}\rho(\overrightarrow{e})=0$.
If $(i_1, i_2) \in A'$, then we must have $(i_2, i_3) \in A'$ for some $i_3 \in V(G)$.
Continuing this process, eventually we obtain a directed cycle of $G$.
In general, if $G$ has a nonhomogeneous cycle (a directed cycle is a nonhomogeneous cycle),
then the dimension of $P_G$ is $|V(G)|-1$ (see \cite[Proposition 1.3]{Higashitani}).
Hence we have $|V(G)|=n+1$.
Since $G$ has at most $n(n+1)$ arrows, $P_G$ has at most $n(n+1)$ vertices.
On the other hand, $P_B$ has $2^{n+1}-2$ vertices.
Thus we have the inequality $2^{n+1}-2 \leq n(n+1)$,
but this inequality does not hold for $n \geq 3$. This is a contradiction.
Thus we proved the proposition.
\end{proof}

\begin{example}
There also exists a reflexive polytope associated to a finite directed graph
that cannot be obtained from any building set.
Let $G$ be the finite directed graph defined by
\begin{equation*}
V(G)=\{1, 2, 3, 4\},\quad
A(G)=\{(1, 2), (2, 3), (3, 1), (1, 4), (4, 3)\}.
\end{equation*}
Then $P_G$ cannot be obtained from any building set.
$P_G$ is a reflexive 3-polytope with six lattice points.
However, there are only three types of reflexive 3-polytopes with six lattice points
that are obtained from building sets.
They are realized by the following building sets:
\begin{align*}
&\{\{1\}, \{2\}, \{3\}, \{4\}, \{1, 2\}, \{1, 2, 3, 4\}\},\\
&\{\{1\}, \{2\}, \{3\}, \{4\}, \{1, 2, 3\}, \{1, 2, 3, 4\}\},\\
&\{\{1\}, \{2\}, \{3\}, \{1, 2, 3\}, \{4\}, \{5\}, \{4, 5\}\}.
\end{align*}
All the building sets yield reflexive polytopes not equivalent to $P_G$.
\begin{figure}[htbp]
\begin{center}
\includegraphics[width=2cm]{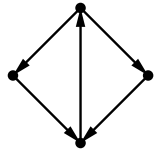}
\caption{the directed graph $G$ whose reflexive polytope cannot
be obtained from any building set.}
\end{center}
\end{figure}
\end{example}

\section{Proof of Theorem \ref{main}}

First we introduce some notation.

\begin{definition}
Let $B$ be a building set on $S$.
\begin{enumerate}
\item We denote by $\mathcal{N}(B)_{\rm max}$
the set of all maximal (by inclusion) nested sets of $B$.
$\mathcal{N}(B)_{\rm max}$ is a subset of $\mathcal{N}(B)$.
\item For $C \in B \setminus B_{\rm max}$, we call
\begin{equation*}
\mathcal{N}(B)_C=\{N \subset (B \setminus B_{\rm max}) \setminus \{C\} \mid
N \cup \{C\} \in \mathcal{N}(B)\}
\end{equation*}
the {\it link} of $C$ in $\mathcal{N}(B)$.
$\mathcal{N}(B)_C$ is an abstract simplicial complex on
\begin{equation*}
\{I \in (B \setminus B_{\rm max}) \setminus \{C\} \mid
\{I, C\} \in \mathcal{N}(B)\}.
\end{equation*}
\item For a nonempty proper subset $C$ of $S$, we call
\begin{equation*}
C \setminus B=\{I \subset S \setminus C \mid I \ne \emptyset;
I \in B \mbox{ or } C \cup I \in B\}
\end{equation*}
the {\it contraction} of $C$ from $B$.
$C \setminus B$ is a building set on $S \setminus C$.
\end{enumerate}
\end{definition}

The {\it symmetric difference} of two sets $X$ and $Y$
is defined by $X \triangle Y=(X \cup Y) \setminus (X \cap Y)$.

\begin{lemma}\label{lema}
Let $B$ be a connected building set on $S$
and let $I_1, I_2 \in B$ with $I_1 \cap I_2 \ne \emptyset,
I_1 \not\subset I_2, I_2 \not\subset I_1$ and $|I_1 \triangle I_2| \geq 3$.
Suppose that
\begin{align*}
&i_1 \in I_1 \setminus I_2, i_2 \in I_2 \setminus I_1,\\
&N \in \mathcal{N}(B|_{I_1 \cap I_2})_{\rm max},
N' \in \mathcal{N}(B|_{(I_1 \triangle I_2) \setminus \{i_1, i_2\}})_{\rm max}
\end{align*}
such that
\begin{equation}\label{I_k}
\{I_k\} \cup N \cup (B|_{I_1 \cap I_2})_{\rm max} \cup N' \cup
(B|_{(I_1 \triangle I_2) \setminus \{i_1, i_2\}})_{\rm max}
\end{equation}
is not a nested set of $B$ for some $k=1, 2$.
Then there exist $I'_1, I'_2 \in B$ such that $I'_1 \supset I_1, I'_2 \supset I_2,
i_1 \in I'_1 \setminus I'_2, i_2 \in I'_2 \setminus I'_1,
I'_1 \cap I'_2 \supsetneq I_1 \cap I_2$ and $I'_1 \cup I'_2=I_1 \cup I_2$.
\end{lemma}

\begin{proof}
The proof is similar to a part of the proof of \cite[Lemma 3.4 (1)]{Suyama1611}.

Without loss of generality, we may assume $k=1$.
Note that $\{I_1\} \cup N \cup (B|_{I_1 \cap I_2})_{\rm max}$
and $N' \cup (B|_{(I_1 \triangle I_2) \setminus \{i_1, i_2\}})_{\rm max}$
are nested sets of $B$.
Thus (\ref{I_k}) falls into the following three cases.

{\it Case 1}. Suppose that (\ref{I_k})
does not satisfy the condition (1) in Definition \ref{nested}.
Then there exist
\begin{equation*}
K \in \{I_1\} \cup N \cup (B|_{I_1 \cap I_2})_{\rm max},\quad
L \in N' \cup (B|_{(I_1 \triangle I_2) \setminus \{i_1, i_2\}})_{\rm max}
\end{equation*}
such that $K \not\subset L, L \not\subset K$ and $K \cap L \ne \emptyset$.
If $K \in N \cup (B|_{I_1 \cap I_2})_{\rm max}$, then $K \cap L=\emptyset$, a contradiction.
Thus we must have $K=I_1$. Then $I_1 \cup L \in B$.
We put $I'_1=I_1 \cup L$ and $I'_2=I_2$.
Since $L \subset I_1 \triangle I_2$,
it follows that $L \setminus I_1 \subset (I'_1 \cap I'_2) \setminus (I_1 \cap I_2)$.
Thus $I'_1 \cap I'_2 \supsetneq I_1 \cap I_2$.

{\it Case 2}. Suppose that (\ref{I_k})
does not satisfy the condition (2) in Definition \ref{nested},
and there exist
\begin{equation*}
K_1, \ldots, K_r \in N \cup (B|_{I_1 \cap I_2})_{\rm max},\quad
L_1, \ldots, L_s \in N' \cup (B|_{(I_1 \triangle I_2) \setminus \{i_1, i_2\}})_{\rm max}
\end{equation*}
for $r, s \geq 1$ such that $K_1, \ldots, K_r, L_1, \ldots, L_s$ are pairwise disjoint
and $K_1 \cup \cdots \cup K_r \cup L_1 \cup \cdots \cup L_s \in B$.
Then we have $I_k \cup L_1 \cup \cdots \cup L_s \in B$ for each $k=1, 2$.
We put $I'_k=I_k \cup L_1 \cup \cdots \cup L_s$ for $k=1, 2$.
$L_1 \cup \cdots \cup L_s \subset I_1 \triangle I_2$ implies
$I_k \subsetneq I'_k$ for some $k=1, 2$.
Since $I'_k \setminus I_k \subset (I'_1 \cap I'_2) \setminus (I_1 \cap I_2)$,
we have $I'_1 \cap I'_2 \supsetneq I_1 \cap I_2$.

{\it Case 3}. Suppose that (\ref{I_k})
does not satisfy the condition (2) in Definition \ref{nested},
and there exist
\begin{equation*}
L_1, \ldots, L_s \in N' \cup (B|_{(I_1 \triangle I_2) \setminus \{i_1, i_2\}})_{\rm max}
\end{equation*}
such that $I_1, L_1, \ldots, L_s$ are pairwise disjoint
and $I_1 \cup L_1 \cup \cdots \cup L_s \in B$.
We put $I'_1=I_1 \cup L_1 \cup \cdots \cup L_s$ and $I'_2=I_2$.
Since $L_1 \cup \cdots \cup L_s \subset I_2$,
it follows that
$L_1 \cup \cdots \cup L_s \subset (I'_1 \cap I'_2) \setminus (I_1 \cap I_2)$.
Thus $I'_1 \cap I'_2 \supsetneq I_1 \cap I_2$.

In every case, we have $i_1 \in I'_1 \setminus I'_2, i_2 \in I'_2 \setminus I'_1$
and $I'_1 \cup I'_2=I_1 \cup I_2$.
This completes the proof.
\end{proof}

Lemmas \ref{lemb} and \ref{lemc} play key roles in the proof of Theorem \ref{main}.

\begin{lemma}\label{lemb}
Let $B$ be a connected building set on $S$
and let $I_1, I_2 \in B$ with $I_1 \cap I_2 \ne \emptyset,
I_1 \not\subset I_2, I_2 \not\subset I_1$ and $I_1 \cap I_2 \notin B$.
Then there exist
\begin{align*}
&J_1, J_2 \in B, j_1 \in J_1 \setminus J_2, j_2 \in J_2 \setminus J_1,\\
&N \in \mathcal{N}(B|_{J_1 \cap J_2})_{\rm max},
N' \in \mathcal{N}(B|_{(J_1 \triangle J_2) \setminus \{j_1, j_2\}})_{\rm max}
\end{align*}
such that $J_1 \cap J_2 \ne \emptyset, J_1 \cap J_2 \notin B,
J_1 \cup J_2 \subset I_1 \cup I_2$ and
\begin{equation*}
\{J_k\} \cup N \cup (B|_{J_1 \cap J_2})_{\rm max} \cup N' \cup
(B|_{(J_1 \triangle J_2) \setminus \{j_1, j_2\}})_{\rm max}
\end{equation*}
is a nested set of $B$ for each $k=1, 2$.
If $J_1 \triangle J_2=\{j_1, j_2\}$,
then $N'$ and $(B|_{(J_1 \triangle J_2) \setminus \{j_1, j_2\}})_{\rm max}$
are understood to be empty.
\end{lemma}

\begin{proof}
We use induction on $|I_1 \triangle I_2|$.
We have $|I_1 \triangle I_2| \geq 2$.
Suppose $|I_1 \triangle I_2|=2$. We put $J_1=I_1$ and $J_2=I_2$.
Then $J_1 \cap J_2 \ne \emptyset, J_1 \cap J_2 \notin B$
and $J_1 \cup J_2=I_1 \cup I_2$.
We pick $N \in \mathcal{N}(B|_{J_1 \cap J_2})_{\rm max}$.
Then $\{J_k\} \cup N \cup (B|_{J_1 \cap J_2})_{\rm max}$
is a nested set of $B$ for each $k=1, 2$.

Suppose $|I_1 \triangle I_2| \geq 3$.
We pick $i_1 \in I_1 \setminus I_2, i_2 \in I_2 \setminus I_1,
N \in \mathcal{N}(B|_{I_1 \cap I_2})_{\rm max}$ and 
$N' \in \mathcal{N}(B|_{(I_1 \triangle I_2) \setminus \{i_1, i_2\}})_{\rm max}$. If
\begin{equation*}
\{I_k\} \cup N \cup (B|_{I_1 \cap I_2})_{\rm max} \cup N' \cup
(B|_{(I_1 \triangle I_2) \setminus \{i_1, i_2\}})_{\rm max}
\end{equation*}
is a nested set of $B$ for each $k=1, 2$, then there is nothing to prove.
Otherwise, by Lemma \ref{lema},
there exist $I'_1, I'_2 \in B$ such that $I'_1 \supset I_1, I'_2 \supset I_2,
i_1 \in I'_1 \setminus I'_2, i_2 \in I'_2 \setminus I'_1,
I'_1 \cap I'_2 \supsetneq I_1 \cap I_2$ and $I'_1 \cup I'_2=I_1 \cup I_2$.

{\it Case 1}. Suppose $I'_1 \cap I'_2 \notin B$. We have
$|I'_1 \triangle I'_2|=|I'_1 \cup I'_2|-|I'_1 \cap I'_2|
<|I_1 \cup I_2|-|I_1 \cap I_2|=|I_1 \triangle I_2|$.
By the hypothesis of induction, there exist
\begin{align*}
&J_1, J_2 \in B, j_1 \in J_1 \setminus J_2, j_2 \in J_2 \setminus J_1,\\
&N \in \mathcal{N}(B|_{J_1 \cap J_2})_{\rm max},
N' \in \mathcal{N}(B|_{(J_1 \triangle J_2) \setminus \{j_1, j_2\}})_{\rm max}
\end{align*}
such that $J_1 \cap J_2 \ne \emptyset, J_1 \cap J_2 \notin B,
J_1 \cup J_2 \subset I'_1 \cup I'_2=I_1 \cup I_2$ and
\begin{equation*}
\{J_k\} \cup N \cup (B|_{J_1 \cap J_2})_{\rm max} \cup N' \cup
(B|_{(J_1 \triangle J_2) \setminus \{j_1, j_2\}})_{\rm max}
\end{equation*}
is a nested set of $B$ for each $k=1, 2$.

{\it Case 2}. Suppose $I'_1 \cap I'_2 \in B$. We may assume that $I_1 \subsetneq I'_1$.

{\it Subcase 2.1}. Suppose $I'_1 \cap I_2 \in B$.
We put $I''_1=I_1$ and $I''_2=I'_1 \cap I_2$.
Then we have $I''_1 \cap I''_2=I_1 \cap I_2 \notin B, i_1 \in I''_1 \setminus I''_2$
and $I'_1 \setminus I_1 \subset I''_2 \setminus I''_1$.
Since $i_2 \in (I_1 \cup I_2) \setminus (I''_1 \cup I''_2)$,
we have $I''_1 \cup I''_2 \subsetneq I_1 \cup I_2$.

{\it Subcase 2.2}. Suppose $I'_1 \cap I_2 \notin B$.
We put $I''_1=I'_1$ and $I''_2=I_2$.
Then we have $I''_1 \cap I''_2=I'_1 \cap I_2 \notin B,
i_1 \in I''_1 \setminus I''_2, i_2 \in I''_2 \setminus I''_1$
and $I''_1 \cup I''_2=I'_1 \cup I_2=I_1 \cup I_2$.
Since $I'_1 \setminus I_1 \subset (I''_1 \cap I''_2) \setminus (I_1 \cap I_2)$,
we have $I_1 \cap I_2 \subsetneq I''_1 \cap I''_2$.

In every subcase, we have
$|I''_1 \triangle I''_2|=|I''_1 \cup I''_2|-|I''_1 \cap I''_2|
<|I_1 \cup I_2|-|I_1 \cap I_2|=|I_1 \triangle I_2|$.
By the hypothesis of induction,
there exist
\begin{align*}
&J_1, J_2 \in B, j_1 \in J_1 \setminus J_2, j_2 \in J_2 \setminus J_1,\\
&N \in \mathcal{N}(B|_{J_1 \cap J_2})_{\rm max},
N' \in \mathcal{N}(B|_{(J_1 \triangle J_2) \setminus \{j_1, j_2\}})_{\rm max}
\end{align*}
such that $J_1 \cap J_2 \ne \emptyset, J_1 \cap J_2 \notin B,
J_1 \cup J_2 \subset I''_1 \cup I''_2 \subset I_1 \cup I_2$ and
\begin{equation*}
\{J_k\} \cup N \cup (B|_{J_1 \cap J_2})_{\rm max} \cup N' \cup
(B|_{(J_1 \triangle J_2) \setminus \{j_1, j_2\}})_{\rm max}
\end{equation*}
is a nested set of $B$ for each $k=1, 2$.

Therefore the assertion holds for $|I_1 \triangle I_2|$.
\end{proof}

\begin{example}
Let $S=\{1, 2, 3, 4, 5, 6\}$ and
\begin{align*}
B&=\{\{1\}, \{2\}, \{3\}, \{4\}, \{5\}, \{6\}, \{1, 2, 3, 4\}, \{2, 3, 4, 5\}, \{3, 4, 5, 6\},\\
&\{1, 2, 3, 4, 5\}, \{2, 3, 4, 5, 6\}, \{1, 2, 3, 4, 5, 6\}\}.
\end{align*}
Let us consider $I_1=\{1, 2, 3, 4\}$ and $I_2=\{3, 4, 5, 6\}$.
We pick $i_1=1$ and $i_2=6$.
Then
\begin{equation*}
B|_{I_1 \cap I_2}=\{\{3\}, \{4\}\},\quad
B|_{(I_1 \triangle I_2) \setminus \{i_1, i_2\}}=\{\{2\}, \{5\}\}.
\end{equation*}
The only maximal nested set of each is the empty set. However,
\begin{align*}
&\{I_1\} \cup \emptyset \cup (B|_{I_1 \cap I_2})_{\rm max}
\cup \emptyset \cup(B|_{(I_1 \triangle I_2) \setminus \{i_1, i_2\}})_{\rm max}\\
&=\{\{1, 2, 3, 4\}, \{3\}, \{4\}, \{2\}, \{5\}\}
\end{align*}
is not a nested set because
$\{3\} \cup \{4\} \cup \{2\} \cup \{5\}=\{2, 3, 4, 5\} \in B$
(Lemma \ref{lema}, Case 2).
Thus we put
\begin{equation*}
I^{(1)}_1=I_1 \cup \{2, 3, 4, 5\}=\{1, 2, 3, 4, 5\},\quad
I^{(1)}_2=I_2 \cup \{2, 3, 4, 5\}=\{2, 3, 4, 5, 6\}.
\end{equation*}
We have $I^{(1)}_1 \cap I^{(1)}_2=\{2, 3, 4, 5\} \in B$ (Lemma \ref{lemb}, Case 2)
and $I_1 \subsetneq I^{(1)}_1$.
Since $I^{(1)}_1 \cap I_2=\{3, 4, 5\} \notin B$ (Subcase 2.2), we put
\begin{equation*}
I^{(2)}_1=I^{(1)}_1=\{1, 2, 3, 4, 5\},\quad
I^{(2)}_2=I_2=\{3, 4, 5, 6\}.
\end{equation*}
We pick $i^{(2)}_1=1$ and $i^{(2)}_2=6$. Then
\begin{equation*}
B|_{I^{(2)}_1 \cap I^{(2)}_2}=\{\{3\}, \{4\}, \{5\}\},\quad
B|_{(I^{(2)}_1 \triangle I^{(2)}_2) \setminus \{i^{(2)}_1, i^{(2)}_2\}}=\{\{2\}\}.
\end{equation*}
The only maximal nested set of each is the empty set.
\begin{align*}
&\{I^{(2)}_1\} \cup \emptyset \cup (B|_{I^{(2)}_1 \cap I^{(2)}_2})_{\rm max}
\cup \emptyset \cup(B|_{(I^{(2)}_1 \triangle I^{(2)}_2)
\setminus \{i^{(2)}_1, i^{(2)}_2\}})_{\rm max}\\
&=\{\{1, 2, 3, 4, 5\}, \{3\}, \{4\}, \{5\}, \{2\}\}
\end{align*}
is not a nested set because
$\{3\} \cup \{4\} \cup \{5\} \cup \{2\}=\{2, 3, 4, 5\} \in B$
(Lemma \ref{lema}, Case 2).
Thus we put
\begin{equation*}
I^{(3)}_1=I^{(2)}_1 \cup \{2, 3, 4, 5\}=\{1, 2, 3, 4, 5\},\quad
I^{(3)}_2=I^{(2)}_2 \cup \{2, 3, 4, 5\}=\{2, 3, 4, 5, 6\}.
\end{equation*}
We have $I^{(3)}_1 \cap I^{(3)}_2=\{2, 3, 4, 5\} \in B$ (Lemma \ref{lemb}, Case 2)
and $I^{(2)}_2 \subsetneq I^{(3)}_2$.
Since $I^{(2)}_1 \cap I^{(3)}_2=\{2, 3, 4, 5\} \in B$ (Subcase 2.1), we put
\begin{equation*}
I^{(4)}_1=I^{(2)}_1 \cap I^{(3)}_2=\{2, 3, 4, 5\},\quad
I^{(4)}_2=I^{(2)}_2=\{3, 4, 5, 6\}.
\end{equation*}
Then
\begin{equation*}
B|_{I^{(4)}_1 \cap I^{(4)}_2}=\{\{3\}, \{4\}, \{5\}\},\quad
|I^{(4)}_1 \triangle I^{(4)}_2|=2.
\end{equation*}
The only maximal nested set of $B|_{I^{(4)}_1 \cap I^{(4)}_2}$ is the empty set and
\begin{align*}
\{I^{(4)}_1\} \cup \emptyset \cup (B|_{I^{(4)}_1 \cap I^{(4)}_2})_{\rm max}
&=\{\{2, 3, 4, 5\}, \{3\}, \{4\}, \{5\}\},\\
\{I^{(4)}_2\} \cup \emptyset \cup (B|_{I^{(4)}_1 \cap I^{(4)}_2})_{\rm max}
&=\{\{3, 4, 5, 6\}, \{3\}, \{4\}, \{5\}\}
\end{align*}
are nested sets of $B$.
\end{example}

\begin{lemma}\label{lemc}
Let $B$ be a connected building set on $S$
and let $I_1, I_2 \in B$ with $I_1 \cap I_2 \ne \emptyset,
I_1 \not\subset I_2, I_2 \not\subset I_1$ and $|(B|_{I_1 \cap I_2})_{\rm max}| \geq 3$.
Then there exist
\begin{align*}
&J_1, J_2 \in B, j_1 \in J_1 \setminus J_2, j_2 \in J_2 \setminus J_1,\\
&N \in \mathcal{N}(B|_{J_1 \cap J_2})_{\rm max},
N' \in \mathcal{N}(B|_{(J_1 \triangle J_2) \setminus \{j_1, j_2\}})_{\rm max}
\end{align*}
such that $J_1 \cap J_2 \ne \emptyset, J_1 \cap J_2 \notin B$ and
\begin{equation*}
\{J_k\} \cup N \cup (B|_{J_1 \cap J_2})_{\rm max} \cup N' \cup
(B|_{(J_1 \triangle J_2) \setminus \{j_1, j_2\}})_{\rm max}
\end{equation*}
is a nested set of $B$ for each $k=1, 2$.
Furthermore, we have $J_1 \cup J_2 \subsetneq I_1 \cup I_2$
or $|(B|_{J_1 \cap J_2})_{\rm max}| \geq 3$.
If $J_1 \triangle J_2=\{j_1, j_2\}$,
then $N'$ and $(B|_{(J_1 \triangle J_2) \setminus \{j_1, j_2\}})_{\rm max}$
are understood to be empty.
\end{lemma}

\begin{proof}
We use induction on $|I_1 \triangle I_2|$.
We have $|I_1 \triangle I_2| \geq 2$.
Suppose $|I_1 \triangle I_2|=2$. We put $J_1=I_1$ and $J_2=I_2$.
Then $J_1 \cap J_2 \ne \emptyset$
and $|(B|_{J_1 \cap J_2})_{\rm max}| \geq 3$.
We pick $N \in \mathcal{N}(B|_{J_1 \cap J_2})_{\rm max}$.
Then $\{J_k\} \cup N \cup (B|_{J_1 \cap J_2})_{\rm max}$
is a nested set of $B$ for each $k=1, 2$.

Suppose $|I_1 \triangle I_2| \geq 3$.
We pick $i_1 \in I_1 \setminus I_2, i_2 \in I_2 \setminus I_1,
N \in \mathcal{N}(B|_{I_1 \cap I_2})_{\rm max}$ and 
$N' \in \mathcal{N}(B|_{(I_1 \triangle I_2) \setminus \{i_1, i_2\}})_{\rm max}$. If
\begin{equation*}
\{I_k\} \cup N \cup (B|_{I_1 \cap I_2})_{\rm max} \cup N' \cup
(B|_{(I_1 \triangle I_2) \setminus \{i_1, i_2\}})_{\rm max}
\end{equation*}
is a nested set of $B$ for each $k=1, 2$, then there is nothing to prove.
Otherwise, by Lemma \ref{lema},
there exist $I'_1, I'_2 \in B$ such that $I'_1 \supset I_1, I'_2 \supset I_2,
i_1 \in I'_1 \setminus I'_2, i_2 \in I'_2 \setminus I'_1,
I'_1 \cap I'_2 \supsetneq I_1 \cap I_2$ and $I'_1 \cup I'_2=I_1 \cup I_2$.

{\it Case 1}. Suppose $|(B|_{I'_1 \cap I'_2})_{\rm max}| \geq 3$.
We have $|I'_1 \triangle I'_2|=|I'_1 \cup I'_2|-|I'_1 \cap I'_2|
<|I_1 \cup I_2|-|I_1 \cap I_2|=|I_1 \triangle I_2|$.
By the hypothesis of induction, there exist
\begin{align*}
&J_1, J_2 \in B, j_1 \in J_1 \setminus J_2, j_2 \in J_2 \setminus J_1,\\
&N \in \mathcal{N}(B|_{J_1 \cap J_2})_{\rm max},
N' \in \mathcal{N}(B|_{(J_1 \triangle J_2) \setminus \{j_1, j_2\}})_{\rm max}
\end{align*}
such that $J_1 \cap J_2 \ne \emptyset, J_1 \cap J_2 \notin B$ and
\begin{equation*}
\{J_k\} \cup N \cup (B|_{J_1 \cap J_2})_{\rm max} \cup N' \cup
(B|_{(J_1 \triangle J_2) \setminus \{j_1, j_2\}})_{\rm max}
\end{equation*}
is a nested set of $B$ for each $k=1, 2$.
Furthermore, we have $J_1 \cup J_2 \subsetneq I'_1 \cup I'_2=I_1 \cup I_2$
or $|(B|_{J_1 \cap J_2})_{\rm max}| \geq 3$.

{\it Case 2}. Suppose $|(B|_{I'_1 \cap I'_2})_{\rm max}| \leq 2$.
For any $K \in (B|_{I_1 \cap I_2})_{\rm max}$,
there exists unique $L_K \in (B|_{I'_1 \cap I'_2})_{\rm max}$ such that $K \subset L_K$.
Hence there exists $L \in (B|_{I'_1 \cap I'_2})_{\rm max}$
that contains more than one element of $ (B|_{I_1 \cap I_2})_{\rm max}$.
Let $K_1, \ldots, K_r$ be all elements of $(B|_{I_1 \cap I_2})_{\rm max}$
contained in $L$.
Note that $I_1 \cap I_2 \cap L$ is the disjoint union of $K_1, \ldots, K_r$.
If $L \subset I_1 \cap I_2$,
then $B \ni L=I_1 \cap I_2 \cap L=K_1 \cup \cdots \cup K_r \notin B$,
a contradiction. Thus $L \not\subset I_1 \cap I_2$.
We may assume $L \not\subset I_1$.

{\it Subcase 2.1}. Suppose $I_1 \cap L \in B$.
If $L \subset I_2$, then
$B \ni I_1 \cap L=I_1 \cap I_2 \cap L=K_1 \cup \cdots \cup K_r \notin B$,
a contradiction. Thus $L \not\subset I_2$.
We put $I''_1=I_1 \cap L$ and $I''_2=I_2$.
Then we have $I''_1 \cap I''_2=I_1 \cap I_2 \cap L \notin B,
L \setminus I_2 \subset I''_1 \setminus I''_2$
and $i_2 \in I''_2 \setminus I''_1$.
Since $i_1 \in (I_1 \cup I_2) \setminus (I''_1 \cup I''_2)$,
we have $I''_1 \cup I''_2 \subsetneq I_1 \cup I_2$.

{\it Subcase 2.2}. Suppose $I_1 \cap L \notin B$. We put $I''_1=I_1$ and $I''_2=L$.
Then we have $I''_1 \cap I''_2=I_1 \cap L \notin B, i_1 \in I''_1 \setminus I''_2$
and $I''_2 \setminus I''_1=L \setminus I_1 \ne \emptyset$.
Since $i_2 \in (I_1 \cup I_2) \setminus (I''_1 \cup I''_2)$,
we have $I''_1 \cup I''_2 \subsetneq I_1 \cup I_2$.

In every subcase, by Lemma \ref{lemb}, there exist
\begin{align*}
&J_1, J_2 \in B, j_1 \in J_1 \setminus J_2, j_2 \in J_2 \setminus J_1,\\
&N \in \mathcal{N}(B|_{J_1 \cap J_2})_{\rm max},
N' \in \mathcal{N}(B|_{(J_1 \triangle J_2) \setminus \{j_1, j_2\}})_{\rm max}
\end{align*}
such that
$J_1 \cap J_2 \ne \emptyset, J_1 \cap J_2 \notin B,
J_1 \cup J_2 \subset I''_1 \cup I''_2 \subsetneq I_1 \cup I_2$ and
\begin{equation*}
\{J_k\} \cup N \cup (B|_{J_1 \cap J_2})_{\rm max} \cup N' \cup
(B|_{(J_1 \triangle J_2) \setminus \{j_1, j_2\}})_{\rm max}
\end{equation*}
is a nested set of $B$ for each $k=1, 2$.

Therefore the assertion holds for $|I_1 \triangle I_2|$.
\end{proof}

\begin{example}
Let $S=\{1, 2, 3, 4, 5, 6, 7\}$ and
\begin{align*}
B&=\{\{1\}, \{2\}, \{3\}, \{4\}, \{5\}, \{6\}, \{7\},
\{2, 4, 6\}, \{2, 3, 4, 5\}, \{1, 2, 3, 4, 5\},\\
&\{2, 3, 4, 5, 6\}, \{3, 4, 5, 6, 7\},
\{1, 2, 3, 4, 5, 6\}, \{2, 3, 4, 5, 6, 7\}, \{1, 2, 3, 4, 5, 6, 7\}\}.
\end{align*}
Let us consider $I_1=\{1, 2, 3, 4, 5\}$ and $I_2=\{3, 4, 5, 6, 7\}$.
We pick $i_1=1$ and $i_2=7$. Then
\begin{equation*}
B|_{I_1 \cap I_2}=\{\{3\}, \{4\}, \{5\}\},\quad
B|_{(I_1 \triangle I_2) \setminus \{i_1, i_2\}}=\{\{2\}, \{6\}\}.
\end{equation*}
The only maximal nested set of each is the empty set. However,
\begin{align*}
&\{I_1\} \cup \emptyset \cup (B|_{I_1 \cap I_2})_{\rm max}
\cup \emptyset \cup(B|_{(I_1 \triangle I_2) \setminus \{i_1, i_2\}})_{\rm max}\\
&=\{\{1, 2, 3, 4, 5\}, \{3\}, \{4\}, \{5\}, \{2\}, \{6\}\}
\end{align*}
is not a nested set because
$\{4\} \cup \{2\} \cup \{6\}=\{2, 4, 6\} \in B$ (Lemma \ref{lema}, Case 2).
Thus we put
\begin{equation*}
I^{(1)}_1=I_1 \cup \{2, 4, 6\}=\{1, 2, 3, 4, 5, 6\},\quad
I^{(1)}_2=I_2 \cup \{2, 4, 6\}=\{2, 3, 4, 5, 6, 7\}.
\end{equation*}
We have $I^{(1)}_1 \cap I^{(1)}_2=\{2, 3, 4, 5, 6\} \in B$
(Lemma \ref{lemc}, Case 2) and $L=\{2, 3, 4, 5, 6\}$.
Since $L \not\subset I_1$ and $I_1 \cap L=\{2, 3, 4, 5\} \in B$ (Subcase 2.1), we put
\begin{equation*}
I^{(2)}_1=I_1 \cap L=\{2, 3, 4, 5\},\quad
I^{(2)}_2=I_2=\{3, 4, 5, 6, 7\}.
\end{equation*}
We pick $i^{(2)}_1=2$ and $i^{(2)}_2=7$. Then
\begin{equation*}
B|_{I^{(2)}_1 \cap I^{(2)}_2}=\{\{3\}, \{4\}, \{5\}\},\quad
B|_{(I^{(2)}_1 \triangle I^{(2)}_2) \setminus \{i^{(2)}_1, i^{(2)}_2\}}=\{\{6\}\}.
\end{equation*}
The only maximal nested set of each is the empty set.
\begin{align*}
&\{I^{(2)}_1\} \cup \emptyset \cup (B|_{I^{(2)}_1 \cap I^{(2)}_2})_{\rm max}
\cup \emptyset \cup(B|_{(I^{(2)}_1 \triangle I^{(2)}_2)
\setminus \{i^{(2)}_1, i^{(2)}_2\}})_{\rm max}\\
&=\{\{2, 3, 4, 5\}, \{3\}, \{4\}, \{5\}, \{6\}\}
\end{align*}
is not a nested set because
$\{2, 3, 4, 5\} \cup \{6\}=\{2, 3, 4, 5, 6\} \in B$ (Lemma \ref{lema}, Case 3).
Thus we put
\begin{equation*}
I^{(3)}_1=\{2, 3, 4, 5, 6\},\quad
I^{(3)}_2=I^{(2)}_2=\{3, 4, 5, 6, 7\}.
\end{equation*}
We have $I^{(3)}_1 \cap I^{(3)}_2=\{3, 4, 5, 6\} \notin B$ (Lemma \ref{lemb}, Case 1) and
\begin{equation*}
B|_{I^{(3)}_1 \cap I^{(3)}_2}=\{\{3\}, \{4\}, \{5\}, \{6\}\},\quad
|I^{(3)}_1 \triangle I^{(3)}_2|=2.
\end{equation*}
The only maximal nested set of $B|_{I^{(3)}_1 \cap I^{(3)}_2}$ is the empty set and
\begin{align*}
\{I^{(3)}_1\} \cup \emptyset \cup (B|_{I^{(3)}_1 \cap I^{(3)}_2})_{\rm max}
&=\{\{2, 3, 4, 5, 6\}, \{3\}, \{4\}, \{5\}, \{6\}\},\\
\{I^{(3)}_2\} \cup \emptyset \cup (B|_{I^{(3)}_1 \cap I^{(3)}_2})_{\rm max}
&=\{\{3, 4, 5, 6, 7\}, \{3\}, \{4\}, \{5\}, \{6\}\}
\end{align*}
are nested sets of $B$.
Furthermore, we have $I^{(3)}_1 \cup I^{(3)}_2 \subsetneq I_1 \cup I_2$.
\end{example}

\begin{example}
Let $S=\{1, 2, 3, 4, 5, 6, 7\}$ and
\begin{align*}
B&=\{\{1\}, \{2\}, \{3\}, \{4\}, \{5\}, \{6\}, \{7\},
\{2, 6\}, \{4, 5, 6\}, \{2, 4, 5, 6\}, \{1, 2, 3, 4, 5\},\\
&\{2, 3, 4, 5, 6\}, \{3, 4, 5, 6, 7\},
\{1, 2, 3, 4, 5, 6\}, \{2, 3, 4, 5, 6, 7\}, \{1, 2, 3, 4, 5, 6, 7\}\}.
\end{align*}
Let us consider $I_1=\{1, 2, 3, 4, 5\}$ and $I_2=\{3, 4, 5, 6, 7\}$.
We pick $i_1=1$ and $i_2=7$. Then
\begin{equation*}
B|_{I_1 \cap I_2}=\{\{3\}, \{4\}, \{5\}\},\quad
B|_{(I_1 \triangle I_2) \setminus \{i_1, i_2\}}=\{\{2\}, \{6\}, \{2, 6\}\}.
\end{equation*}
The only maximal nested set of $B|_{I_1 \cap I_2}$ is the empty set.
We choose
$N'=\{\{2\}\} \in \mathcal{N}(B|_{(I_1 \triangle I_2) \setminus \{i_1, i_2\}})_{\rm max}$.
\begin{align*}
&\{I_1\} \cup \emptyset \cup (B|_{I_1 \cap I_2})_{\rm max}
\cup N' \cup(B|_{(I_1 \triangle I_2) \setminus \{i_1, i_2\}})_{\rm max}\\
&=\{\{1, 2, 3, 4, 5\}, \{3\}, \{4\}, \{5\}, \{2\}, \{2, 6\}\}
\end{align*}
is not a nested set because
$\{1, 2, 3, 4, 5\} \cap \{2, 6\}=\{2\} \ne \emptyset$ (Lemma \ref{lema}, Case 1).
Thus we put
\begin{equation*}
I^{(1)}_1=I_1 \cup \{2, 6\}=\{1, 2, 3, 4, 5, 6\},\quad
I^{(1)}_2=I_2=\{3, 4, 5, 6, 7\}.
\end{equation*}
We have $I^{(1)}_1 \cap I^{(1)}_2=\{3, 4, 5, 6\}=\{3\} \cup \{4, 5, 6\}$
(Lemma \ref{lemc}, Case 2) and $L=\{4, 5, 6\}$.
Since $L \not\subset I_1$ and $I_1 \cap L=\{4, 5\} \notin B$ (Subcase 2.2), we put
\begin{equation*}
I^{(2)}_1=I_1=\{1, 2, 3, 4, 5\},\quad
I^{(2)}_2=L=\{4, 5, 6\}.
\end{equation*}
We pick $i^{(2)}_1=1$ and $i^{(2)}_2=6$. Then
\begin{equation*}
B|_{I^{(2)}_1 \cap I^{(2)}_2}=\{\{4\}, \{5\}\},\quad
B|_{(I^{(2)}_1 \triangle I^{(2)}_2) \setminus \{i^{(2)}_1, i^{(2)}_2\}}=\{\{2\}, \{3\}\}.
\end{equation*}
The only maximal nested set of each is the empty set.
\begin{align*}
&\{I^{(2)}_2\} \cup \emptyset \cup (B|_{I^{(2)}_1 \cap I^{(2)}_2})_{\rm max}
\cup \emptyset \cup(B|_{(I^{(2)}_1 \triangle I^{(2)}_2)
\setminus \{i^{(2)}_1, i^{(2)}_2\}})_{\rm max}\\
&=\{\{4, 5, 6\}, \{4\}, \{5\}, \{2\}, \{3\}\}
\end{align*}
is not a nested set because
$\{4, 5, 6\} \cup \{2\} \cup \{3\}=\{2, 3, 4, 5, 6\} \in B$ (Lemma \ref{lema}, Case 3).
Thus we put
\begin{equation*}
I^{(3)}_1=I^{(2)}_1=\{1, 2, 3, 4, 5\},\quad
I^{(3)}_2=\{2, 3, 4, 5, 6\}.
\end{equation*}
We have $I^{(3)}_1 \cap I^{(3)}_2=\{2, 3, 4, 5\} \notin B$ (Lemma \ref{lemb}, Case 1) and
\begin{equation*}
B|_{I^{(3)}_1 \cap I^{(3)}_2}=\{\{2\}, \{3\}, \{4\}, \{5\}\},\quad
|I^{(3)}_1 \triangle I^{(3)}_2|=2.
\end{equation*}
The only maximal nested set of $B|_{I^{(3)}_1 \cap I^{(3)}_2}$ is the empty set and
\begin{align*}
\{I^{(3)}_1\} \cup \emptyset \cup (B|_{I^{(3)}_1 \cap I^{(3)}_2})_{\rm max}
&=\{\{1, 2, 3, 4, 5\}, \{2\}, \{3\}, \{4\}, \{5\}\},\\
\{I^{(3)}_2\} \cup \emptyset \cup (B|_{I^{(3)}_1 \cap I^{(3)}_2})_{\rm max}
&=\{\{2, 3, 4, 5, 6\}, \{2\}, \{3\}, \{4\}, \{5\}\}
\end{align*}
are nested sets of $B$.
Furthermore, we have $I^{(3)}_1 \cup I^{(3)}_2 \subsetneq I_1 \cup I_2$.
\end{example}

\begin{proposition}[{\cite[Proposition 3.2]{Zelevinsky}}]\label{link}
Let $B$ be a building set on $S$ and let $C \in B \setminus B_{\rm max}$.
Then the correspondence
\begin{equation*}
I \mapsto \left\{\begin{array}{ll}
I \setminus C & (C \subset I), \\
I & (C \not\subset I) \end{array}\right.
\end{equation*}
induces an isomorphism
$\mathcal{N}(B)_C \rightarrow \mathcal{N}(B|_C \cup (C \setminus B))$
of simplicial complexes.
\end{proposition}

\begin{lemma}\label{enlarge}
Let $J_1, J_2 \in B$ with $J_1 \cap J_2 \ne \emptyset,
J_1 \not\subset J_2, J_2 \not\subset J_1$ and $J_1 \cup J_2 \subsetneq S$.
Let $N'' \in \mathcal{N}(B|_{J_1 \cup J_2})$ such that
$\{J_k\} \cup N'' \in \mathcal{N}(B|_{J_1 \cup J_2})_{\rm max}$
for each $k=1, 2$.
Then there exists $M \in \mathcal{N}(B)$ such that
$\{J_k, J_1 \cup J_2\} \cup N'' \cup M \in \mathcal{N}(B)_{\rm max}$
for each $k=1, 2$.
\end{lemma}

\begin{proof}
We pick $M' \in \mathcal{N}((J_1 \cup J_2) \setminus B)_{\rm max}$. Then
\begin{align*}
\{J_k\} \cup N'' \cup M'
\in \mathcal{N}(B|_{J_1 \cup J_2} \cup ((J_1 \cup J_2) \setminus B))_{\rm max}
\end{align*}
for each $k=1, 2$. Hence by Proposition \ref{link},
there exists $M \in \mathcal{N}(B)$ such that
$ \{J_k\} \cup N'' \cup M$
are maximal simplices of $\mathcal{N}(B)_{J_1 \cup J_2}$. Hence
$\{J_k, J_1 \cup J_2\} \cup N'' \cup M \in \mathcal{N}(B)_{\rm max}$
for each $k=1, 2$.
\end{proof}

\begin{proposition}[{\cite[Proposition 4.5]{Zelevinsky}}]\label{pair}
Let $B$ be a building set on $S$
and let $I_1, I_2 \in B$ with $I_1 \ne I_2$ and $N \in \mathcal{N}(B)$
such that $N \cup \{I_1\}, N \cup \{I_2\} \in \mathcal{N}(B)_{\rm max}$.
Then the following hold:
\begin{enumerate}
\item We have $I_1 \not\subset I_2$ and $I_2 \not\subset I_1$.
\item If $I_1 \cap I_2 \ne \emptyset$, then $(B|_{I_1 \cap I_2})_{\rm max} \subset N$.
\item There exists $\{I_3, \ldots, I_k\} \subset N$ such that
$I_1 \cup I_2, I_3, \ldots, I_k$ are pairwise disjoint
and $I_1 \cup \cdots \cup I_k \in N \cup B_{\rm max}$ ($\{I_3, \ldots, I_k\}$ can be empty).
\end{enumerate}
\end{proposition}

We are now ready to prove Theorem \ref{main}.

\begin{proof}[Proof of Theorem \ref{main}]
The disjoint union of connected building sets yields
the product of toric varieties associated to the connected building sets.
Since the product of nonsingular projective toric varieties is weak Fano if and only if
every factor is weak Fano, it suffices to show that,
for any connected building set $B$ on $S=\{1, \ldots, n+1\}$,
the following are equivalent:
\begin{itemize}
\item[$(1')$] The associated toric variety $X(\Delta(B))$ is weak Fano.
\item[$(2')$] For any $I_1, I_2 \in B$ such that $I_1 \cap I_2 \ne \emptyset,
I_1 \not\subset I_2$ and $I_2 \not\subset I_1$, we have at least one of the following:
\begin{itemize}
\item[$({\rm i}')$] $I_1 \cap I_2 \in B$.
\item[$({\rm ii}')$] $I_1 \cup I_2=S$ and $|(B|_{I_1 \cap I_2})_{\rm max}| \leq 2$.
\end{itemize}
\end{itemize}

$(1') \Rightarrow (2')$:
Let $I_1, I_2 \in B$ such that $I_1 \cap I_2 \ne \emptyset,
I_1 \not\subset I_2, I_2 \not\subset I_1$ and $I_1 \cap I_2 \notin B$.
We show that if $I_1 \cup I_2 \subsetneq S$ or $|(B|_{I_1 \cap I_2})_{\rm max}| \geq 3$,
then the toric variety $X(\Delta(B))$ is not weak Fano.

{\it Case 1}. Suppose $I_1 \cup I_2 \subsetneq S$.
By Lemma \ref{lemb}, there exist
\begin{align*}
&J_1, J_2 \in B, j_1 \in J_1 \setminus J_2, j_2 \in J_2 \setminus J_1,\\
&N \in \mathcal{N}(B|_{J_1 \cap J_2})_{\rm max},
N' \in \mathcal{N}(B|_{(J_1 \triangle J_2) \setminus \{j_1, j_2\}})_{\rm max}
\end{align*}
such that $J_1 \cap J_2 \ne \emptyset, J_1 \cap J_2 \notin B,
J_1 \cup J_2 \subset I_1 \cup I_2 \subsetneq S$ and
\begin{equation}\label{J_k}
\{J_k\} \cup N \cup (B|_{J_1 \cap J_2})_{\rm max} \cup N' \cup
(B|_{(J_1 \triangle J_2) \setminus \{j_1, j_2\}})_{\rm max}
\end{equation}
is a nested set of $B$ for each $k=1, 2$.
Since the cardinality of (\ref{J_k}) is $|J_1 \cup J_2|-1$,
(\ref{J_k}) is a maximal nested set of $B|_{J_1 \cup J_2}$.
By Lemma \ref{enlarge}, there exists $M \in \mathcal{N}(B)$ such that
\begin{equation*}
\{J_k, J_1 \cup J_2\} \cup N \cup (B|_{J_1 \cap J_2})_{\rm max} \cup N' \cup
(B|_{(J_1 \triangle J_2) \setminus \{j_1, j_2\}})_{\rm max} \cup M \in \mathcal{N}(B)_{\rm max}
\end{equation*}
for $k=1, 2$. Let
\begin{equation*}
\tau=\mathbb{R}_{\geq0}(\{J_1 \cup J_2\} \cup N \cup (B|_{J_1 \cap J_2})_{\rm max}
\cup N' \cup (B|_{(J_1 \triangle J_2) \setminus \{j_1, j_2\}})_{\rm max} \cup M).
\end{equation*}
Clearly
\begin{equation*}
e_{J_1}+e_{J_2}-\sum_{C \in (B|_{J_1 \cap J_2})_{\rm max}}e_C-e_{J_1 \cup J_2}=0.
\end{equation*}
Since $|(B|_{J_1 \cap J_2})_{\rm max}| \geq 2$,
Proposition \ref{intersection number} gives
\begin{equation*}
(-K_{X(\Delta(B))}.V(\tau))=2-|(B|_{J_1 \cap J_2})_{\rm max}|-1 \leq 2-2-1=-1.
\end{equation*}
Therefore $X(\Delta(B))$ is not weak Fano by Proposition \ref{weak Fano}.

{\it Case 2}. Suppose that $I_1 \cup I_2=S$ and $|(B|_{I_1 \cap I_2})_{\rm max}| \geq 3$.
By Lemma \ref{lemc}, there exist
\begin{align*}
&J_1, J_2 \in B, j_1 \in J_1 \setminus J_2, j_2 \in J_2 \setminus J_1,\\
&N \in \mathcal{N}(B|_{J_1 \cap J_2})_{\rm max},
N' \in \mathcal{N}(B|_{(J_1 \triangle J_2) \setminus \{j_1, j_2\}})_{\rm max}
\end{align*}
such that $J_1 \cap J_2 \ne \emptyset, J_1 \cap J_2 \notin B$ and
\begin{equation}\label{J_k2}
\{J_k\} \cup N \cup (B|_{J_1 \cap J_2})_{\rm max} \cup N' \cup
(B|_{(J_1 \triangle J_2) \setminus \{j_1, j_2\}})_{\rm max}
\end{equation}
is a nested set of $B$ for each $k=1, 2$.
Furthermore, we have $J_1 \cup J_2 \subsetneq I_1 \cup I_2=S$
or $|(B|_{J_1 \cap J_2})_{\rm max}| \geq 3$.
If $J_1 \cup J_2 \subsetneq S$,
then a similar augment shows that $X(\Delta(B))$ is not weak Fano.
Suppose that $J_1 \cup J_2=S$ and $|(B|_{J_1 \cap J_2})_{\rm max}| \geq 3$.
Then (\ref{J_k2}) is a maximal nested set of $B$. Let
\begin{equation*}
\tau=\mathbb{R}_{\geq0}(N \cup (B|_{J_1 \cap J_2})_{\rm max}
\cup N' \cup (B|_{(J_1 \triangle J_2) \setminus \{j_1, j_2\}})_{\rm max}).
\end{equation*}
Since $e_{J_1 \cup J_2}=e_S=0$, it follows that
\begin{equation*}
e_{J_1}+e_{J_2}-\sum_{C \in (B|_{J_1 \cap J_2})_{\rm max}}e_C=0.
\end{equation*}
Proposition \ref{intersection number} gives
\begin{equation*}
(-K_{X(\Delta(B))}.V(\tau))=2-|(B|_{J_1 \cap J_2})_{\rm max}| \leq 2-3=-1.
\end{equation*}
Therefore $X(\Delta(B))$ is not weak Fano by Proposition \ref{weak Fano}.

$(2') \Rightarrow (1')$: Let $I_1, I_2 \in B$ with $I_1 \ne I_2$ and $N \in \mathcal{N}(B)$
such that $N \cup \{I_1\}, N \cup \{I_2\} \in \mathcal{N}(B)_{\rm max}$.
We need to show that $(-K_{X(\Delta(B))}.V(\mathbb{R}_{\geq0}N)) \geq 0$.

{\it Case 1}. Suppose $I_1 \cap I_2=\emptyset$.
By Proposition \ref{pair} (3),
there exists $\{I_3, \ldots, I_k\} \subset N$ such that
$I_1 \cup I_2, I_3, \ldots, I_k$ are pairwise disjoint
and $I_1 \cup \cdots \cup I_k \in N \cup B_{\rm max}=N \cup \{S\}$. Since
\begin{equation*}
e_{I_1}+e_{I_2}+e_{I_3}+\cdots+e_{I_k}-e_{I_1 \cup \cdots \cup I_k}=0,
\end{equation*}
Proposition \ref{intersection number} gives
\begin{equation*}
(-K_{X(\Delta(B))}.V(\mathbb{R}_{\geq0}N))=\left\{\begin{array}{ll}
k-1 & (I_1 \cup \cdots \cup I_k \in N), \\
k & (I_1 \cup \cdots \cup I_k=S). \end{array}\right.
\end{equation*}
Hence $(-K_{X(\Delta(B))}.V(\mathbb{R}_{\geq0}N)) \geq 1$.

{\it Case 2}. Suppose $I_1 \cap I_2 \ne \emptyset$.
By Proposition \ref{pair} (1),
we have $I_1 \not\subset I_2$ and $I_2 \not\subset I_1$.

$({\rm i}')$ Suppose $I_1 \cap I_2 \in B$.
By Proposition \ref{pair} (2),
we have $\{I_1 \cap I_2\}=(B|_{I_1 \cap I_2})_{\rm max} \subset N$.
By Proposition \ref{pair} (3),
there exists $\{I_3, \ldots, I_k\} \subset N$ such that
$I_1 \cup I_2, I_3, \ldots, I_k$ are pairwise disjoint
and $I_1 \cup \cdots \cup I_k \in N \cup B_{\rm max}=N \cup \{S\}$. Since
\begin{equation*}
e_{I_1}+e_{I_2}-e_{I_1 \cap I_2}+e_{I_3}+\cdots+e_{I_k}-e_{I_1 \cup \cdots \cup I_k}=0,
\end{equation*}
Proposition \ref{intersection number} gives
\begin{equation*}
(-K_{X(\Delta(B))}.V(\mathbb{R}_{\geq0}N))=\left\{\begin{array}{ll}
k-2 & (I_1 \cup \cdots \cup I_k \in N), \\
k-1 & (I_1 \cup \cdots \cup I_k=S). \end{array}\right.
\end{equation*}
Hence $(-K_{X(\Delta(B))}.V(\mathbb{R}_{\geq0}N)) \geq 0$.

$({\rm ii}')$ Suppose that $I_1 \cup I_2=S$ and $|(B|_{I_1 \cap I_2})_{\rm max}| \leq 2$.
By Proposition \ref{pair} (2),
we have $(B|_{I_1 \cap I_2})_{\rm max} \subset N$.
Since $e_{I_1 \cup I_2}=e_S=0$, it follows that
\begin{equation*}
e_{I_1}+e_{I_2}-\sum_{C \in (B|_{I_1 \cap I_2})_{\rm max}}e_C=0.
\end{equation*}
Proposition \ref{intersection number} gives
\begin{equation*}
(-K_{X(\Delta(B))}.V(\mathbb{R}_{\geq0}N))=2-|(B|_{I_1 \cap I_2})_{\rm max}| \geq 2-2=0.
\end{equation*}

Therefore $X(\Delta(B))$ is weak Fano by Proposition \ref{weak Fano}.
This completes the proof of Theorem \ref{main}.
\end{proof}

\end{document}